%% file: UFC.tex
\documentclass{amsart}

\include{UFC-preamble}
\newcommand{\Cs}{\mathrm{C}^*}
\newcommand{\Ws}{\mathrm{W}^*}

\title{Uniqueness of unitary structure for unitarizable fusion categories}

\author{David Reutter}
\address{Max Planck Institute for Mathematics, Bonn}
\email{david.reutter@mpim-bonn.mpg.de}
\urladdr{https://www.davidreutter.com}

\begin{document}

\maketitle	

\begin{abstract}We show that every unitarizable fusion category, and more generally every semisimple $\Cs$-tensor category, admits a unique unitary structure. Our proof is based on a categorified polar decomposition theorem for monoidal equivalences between such categories. We prove analogous results for unitarizable braided fusion categories and module categories.
\end{abstract}

\section*{Introduction}

A unitary fusion category is a fusion category over the complex numbers with a positive $*$-structure. Such unitary structures naturally arise in many applications and constructions \ignore{and classifications }of fusion categories, most notably in the context of operator algebras, subfactor theory, and mathematical physics. Not every fusion category admits a unitary structure; the Yang-Lee category~\cite[Ex 8.18.7, Ex 9.4.6]{EGNO} is a famous example of a non-unitary fusion category. Conversely, a fusion category could, in principle, have more than one unitary structure. Especially in light of the recent powerful operator algebraic classification techniques of unitary fusion categories~\cite{EG1,EG2,Jonessurvey,subfactors,IzumiCuntz}, the question of uniqueness of unitary structure has become a significant open problem~~\cite{mathoverflow, CompleteUnitary, Galindo, PenneysHenriques}. (The explicit statement of Theorem~\ref{thm:functor} appears as Question 2.8 in~\cite{PenneysHenriques}.) A partial answer for weakly group theoretical fusion categories was obtained in~\cite{CompleteUnitary}.

In this paper, we completely address the general question and prove that every unitarizable\footnote{A \emph{unitarizable} fusion category is a fusion category which admits a compatible positive $*$-structure.} fusion category, and more generally every semisimple $\Cs$-tensor category,  admits a unique unitary structure. We also extend our techniques to unitarizable braided fusion categories and unitarizable module categories.

\subsection*{Unitary fusion categories}

\begin{maintheorem}\label{thm:functor} Every monoidal equivalence between unitary fusion categories is monoidally naturally isomorphic to a unitary monoidal equivalence. \\In particular, any unitarizable fusion category admits a unique unitary structure (up to unitary monoidal equivalence).
\end{maintheorem}
Theorem~\ref{thm:functor} is proven in Section~\ref{sec:theorem:functor}. It essentially follows from a categorification of polar decomposition (Corollary~\ref{cor:polarsemisimple}): Every monoidal equivalence between semisimple $\Cs$-categories (with possibly infinitely many simple objects) can be factored as a unitary monoidal equivalence followed by a `positive monoidal auto-equivalence' (see Definition~\ref{def:positivemonoidal}). In particular, every semisimple $\Cs$-tensor category admits a unique $*$-structure. If $\cD$ is a unitary fusion category, it moreover follows from the finiteness of its group of monoidal auto-equivalences that every positive monoidal auto-equivalence of $\cD$ is trivial (Proposition~\ref{prop:functorstep2}), leading to Theorem~\ref{thm:functor}.

When preparing this manuscript, we became aware of an alternative, independent proof of Theorem~\ref{thm:functor} by Carpi, Ciamprone and Pinzari (to appear~\cite{Pinzari}) using the framework of weak quasi Hopf algebras.

Natural isomorphisms between unitary monoidal equivalences behave analogously.
\begin{maintheorem} \label{thm:trafo} Let $\FF,\GG:\cC \to \cD$ be unitary monoidal equivalences between unitary fusion categories.  Then, every monoidal natural isomorphism $\eta:\FF \To \GG$ is unitary.
\end{maintheorem}
Theorem~\ref{thm:trafo} is proven in Section~\ref{sec:theorem:trafo}. It is again a consequence of polar decomposition: Every monoidal natural isomorphism factors into a unitary monoidal natural isomorphism followed by a positive monoidal natural automorphism. It follows from the finiteness of the universal grading group of $\cD$ that there are no non-trivial positive monoidal natural automorphisms.

Adopting notation from~\cite{ENOHomotopy}, Theorems~\ref{thm:functor} and~\ref{thm:trafo} show that the $2$-groupoid $\EQ$ of unitarizable fusion categories, monoidal equivalences and monoidal natural isomorphisms is equivalent to the 2-groupoid $\EQ^\dagger$ of unitary fusion categories, unitary monoidal equivalences and unitary monoidal natural isomorphisms. 
\begin{maincor}\label{cor:Eq} The forgetful $2$-functor $\EQ^\dagger \to \EQ$ is an equivalence. 
\end{maincor}
In a certain sense, the existence of a unitary structure on a fusion category $\cC$ may therefore be thought of as a property of $\cC$ (`unitarizability'), rather than as additional structure.

\begin{mainremark} Since the group of monoidal auto-equivalences of a multifusion category is finite~\cite[Thm 4.15]{ENOHomotopy}, Theorem~\ref{thm:functor} immediately generalizes to multifusion categories. However, Theorem~\ref{thm:trafo} does not hold for multifusion categories: There can be non-trivial positive (and hence non-unitary) monoidal natural transformations between unitary monoidal equivalences of unitary multifusion categories (see e.g.~\cite{PenneysUnitary}).
\end{mainremark}

\subsection*{Unitary braided fusion categories}
The constructions in the proof of Theorem~\ref{thm:functor} are compatible with braidings, allowing us to extend our results to unitary braided fusion categories.

\begin{maintheorem}\label{thm:braiding}Every braided monoidal equivalence $\cA \to \cB$ between unitary braided fusion categories is monoidally naturally isomorphic to a unitary braided monoidal equivalence.
\end{maintheorem}
Theorem~\ref{thm:braiding} is proven in Section~\ref{sec:theorem:braiding}.

By a result of Galindo~\cite[Thm 3.2]{Galindo}, any braiding on a unitary fusion category is unitary. Combining this with Theorem~\ref{thm:braiding} shows that there is a unique unitary braided structure on any braided unitarizable fusion category\footnote{A \emph{braided unitarizable fusion category} is a braided fusion category whose underlying fusion category is unitarizable.}. 

\begin{maincor} Every braided unitarizable fusion category admits a unique unitary braided structure (up to unitary braided monoidal equivalence). 
\end{maincor}

Together with Theorem~\ref{thm:trafo}, Theorem~\ref{thm:braiding} shows that the $2$-groupoid $\EQBR$ of braided unitarizable fusion categories, braided monoidal equivalences and monoidal natural isomorphisms is equivalent to the $2$-groupoid $\EQBR^\dagger$ of unitary braided fusion categories, unitary braided monoidal equivalences and unitary natural isomorphisms.
\begin{maincor}\label{cor:EqBr} The forgetful $2$-functor $\EQBR^\dagger \to \EQBR$ is an equivalence. 
\end{maincor}

\subsection*{Unitary module categories}
The proof of Theorem~\ref{thm:functor} translates almost directly into a proof of uniqueness of the unitary structure on a unitarizable module category.
\begin{maintheorem} \label{thm:modfunctor}Every module equivalence ${}_{\cC}\cM\to {}_{\cC} \cN$ between unitary module categories over unitary fusion categories $\cC$ and $\cD$ is naturally isomorphic, as a module functor, to a unitary module equivalence.
\end{maintheorem}
Theorem~\ref{thm:modfunctor} is proven in Section~\ref{sec:theorem:modfunctor}. Its proof is completely analogous to the proof of Theorem~\ref{thm:functor}, replacing finiteness of the group of monoidal auto-equivalences with finiteness of the group of module auto-equivalences. Again, uniqueness of the unitary structure on a unitarizable module category\footnote{A \emph{unitarizable module category} over a unitary fusion category $\cC$ is a finitely semisimple module category over $\cC$ which admits a compatible positive dagger structure.} is an immediate corollary.
\begin{maincor} \label{cor:uniquemodule}Every unitarizable module category of a unitary fusion category admits a unique unitary structure (up to unitary module equivalence).
\end{maincor}

\begin{mainremark}\label{rem:positivemodule}
The category of $*$-module functors between two given unitary module categories is a finitely semisimple $\Cs$-category. Hence, any module natural isomorphism $\eta: \FF \To \GG$ between unitary module equivalences may be factored into a unitary module natural isomorphism followed by a positive module natural isomorphism. However, there is no analogous statement to Theorem~\ref{thm:trafo}; there are non-trivial positive module natural isomorphisms.
\end{mainremark}
\begin{mainremark}\label{rem:invertiblebimod}
Similar to Corollaries~\ref{cor:Eq} and~\ref{cor:EqBr}, it would be interesting to compare the algebraic Brauer-Picard $2$-groupoid\footnote{By Remark~\ref{rem:positivemodule}, there are non-unitary module natural isomorphisms. Hence, we may only hope for an equivalence between the $2$-truncations $\BrPic$ and $\BrPic^\dagger$ of the algebraic and unitary Brauer-Picard $3$-groupoids $\underline{\BrPic}$ and $\underline{\BrPic}^\dagger$.} $\BrPic$ of unitarizable fusion categories, invertible bimodule categories, and natural isomorphism classes of bimodule equivalences with its unitary counterpart $\BrPic^\dagger$ and prove that the forgetful $2$-functor $\BrPic^\dagger \to \BrPic$ is an equivalence. In fact, by Corollary~\ref{cor:EqBr} and the fact that the unitary Drinfeld center $\cZ^\dagger(\cC)$ of a unitary fusion category \emph{equals}~\cite[Prop 3.1]{Galindo} its ordinary Drinfeld center $\cZ(\cC)$, this would be a direct consequence of the widely expected (but to our knowledge still unproven) unitary version of the equivalence $\BrPic(\cC) \to \EQBR(\cZ(\cC))$ (see~\cite[Thm 1.1]{ENOHomotopy}).
\end{mainremark}

Remark~\ref{rem:invertiblebimod} provides very strong evidence that every \emph{invertible} finitely semisimple bimodule category between unitary fusion categories is unitarizable and hence admits a unique unitary structure. Indeed, note that Theorems~\ref{thm:functor},~\ref{thm:trafo},~\ref{thm:braiding} and~\ref{thm:modfunctor} may all be understood as asserting `unitarizability' of various \emph{invertible} morphisms. For non-invertible morphisms --- such as non-invertible bimodule categories --- the situation is far less clear.
\begin{mainquestion} Does every finitely semisimple module category over a unitary fusion category admit a (by Corollary~\ref{cor:uniquemodule} necessarily unique) unitary structure?
\end{mainquestion}

\subsection*{A cohomological perspective}
Our results can be understood as a generalization of the fact that for a finite group $G$, the homomorphism $H^n(G, U(1)) \to H^n(G, \CC^\times)$ induced from the inclusion $U(1) \hookrightarrow \CC^\times$ is an isomorphism. Indeed, our proof mirrors the following elementary proof of this fact: Let $\omega(g_1,\ldots, g_n)$ denote a $\CC^\times$-valued n-cocycle. Taking absolute values of the $n$-cocycle equation shows that $|\omega(g_1,\ldots, g_n)|$ and $u(g_1,\ldots, g_n):= \omega(g_1,\ldots, g_n)/|\omega(g_1,\ldots, g_n)|$ are $n$-cocycles in $H^n(G,\RR_{>0})$ and $H^n(G,U(1))$, respectively, and that $\omega(g_1,\ldots, g_n) = u(g_1,\ldots, g_n) |\omega(g_1,\ldots, g_n)|$. Moreover, it follows from finiteness of $G$ that $H^n(G, \RR_{>0})$ is zero, and hence that the positive $n$-cocycle $|\omega(g_1,\ldots,g_n)|$ is cohomologous to the trivial cocycle. Therefore, $\omega$ is cohomologous to $u$. 

Our proofs of Theorems~\ref{thm:functor} and~\ref{thm:trafo} proceed analogously, replacing the factorization $\omega= u~ |\omega|$ by a polar decomposition and deducing triviality of the positive part from the finiteness of the group of monoidal auto-equivalences and the universal grading group, respectively. 

More precisely, given a unitary fusion category $\cC$, Theorem~\ref{thm:functor} is a direct consequence of the vanishing of the group $h^2(\cC, \RR_{>0})$ of (natural isomorphism classes of) monoidal auto-equivalences of $\cC$ with underlying identity functor and positive coherence natural isomorphism (see Proposition~\ref{prop:functorstep2}). Similarly, Theorem~\ref{thm:trafo} is a direct consequence of the vanishing of the group $h^1(\cC, \RR_{>0})$ of positive monoidal natural automorphisms of the identity monoidal equivalence $\id_{\cC}: \cC \to \cC$ (see Proposition~\ref{prop:trafostep2}). The appearance of these groups is no coincidence: If $\cC$ is the category of $G$-graded vector spaces, then $h^1(\cC, \RR_{>0})$ and $h^2(\cC, \RR_{>0})$ are precisely the group cohomology groups $H^1(G, \RR_{>0})$ and $H^2(G, \RR_{>0})$. In general, there is no obvious cohomology theory $h^n$ of fusion categories specializing to $h^1$ and $h^2$ at $n=1,2$, but the prominent appearance of these groups nevertheless raises the question of whether our proofs are elementary versions of a more elegant cohomological result.

\subsection*{Generalizing to semisimple $\Cs$-categories and beyond}
In Corollary~\ref{cor:polarsemisimple}, we prove a polar decomposition theorem for semisimple $\Cs$-tensor categories (with possibly infinitely many simple objects): Every monoidal equivalence between semisimple $\Cs$-tensor categories is monoidally naturally isomorphic to the composite of a unitary monoidal equivalence followed by a positive monoidal auto-equivalence. In particular, any semisimple $\Cs$-tensor category admits a unique $\Cs$-structure (even though there might be non-trivial positive monoidal auto-equivalences, see Remark~\ref{rem:positiveauto}). In fact, we expect Corollary~\ref{cor:polarsemisimple} to be more generally true for arbitrary $\Ws$-tensor categories. By Theorem~\ref{thm:functorstep1}, this would follow from a positive answer to the following question:
\begin{mainquestion}\label{q:Wstar}Is every equivalence $F:\cC \to \cD$ between $\Ws$-categories naturally isomorphic to a $*$-equivalence?
\end{mainquestion}
\nid
Note that a positive answer to this question may be understood as a generalization of a theorem of Okayasu~\cite{Okayasu} that any (not necessarily $*$-preserving) isomorphism $A\to B$ of $\Ws$-algebras is of the form $\eta \phi(-) \eta^{-1}:A \to B$, where $\phi:A\to B$ is a $*$-isomorphism and $\eta$ is a positive invertible element of $B$.

\subsection*{Monoidal 2-categories}

Most of the material in this paper was developed using the framework and graphical calculus of the monoidal $2$-category $\tHilb$~\cite{Baez}. However, to keep our presentation as elementary and accessible as possible, we present all proofs in terms of standard $1$-categorical machinery, completely omitting the use of higher category theory.

For the interested reader, we demonstrate in Section~\ref{sec:2Hilb} how some of the more subtle applications of naturality in the proof of Theorem~\ref{thm:functorstep1} arise from simple isotopies in the graphical calculus. Section~\ref{sec:2Hilb} is purely expositional and is not necessary for the technical developments of our results.

\subsection*{Acknowledgements} I am grateful to Andr\'e Henriques and David Penneys for many helpful comments and suggestions on an early version of this manuscript, and to Makoto Yamashita for pointing out that the proof of Theorem~\ref{thm:functorstep1} also applies to semisimple $\Cs$-categories with infinitely many simple objects.

\section{Preliminaries}\label{sec:prelim}

\subsection{$*$-categories}\label{sec:lineardagger}
In the following, a \emph{linear category} is a category enriched in the category $\Vect_{\CC}$ of $\CC$-vector spaces and linear maps, and a \emph{linear functor} is a $\Vect_{\CC}$-enriched functor.

A \emph{$*$-category} is a linear category equipped with a \emph{$*$-structure}, a $\CC$-antilinear, involutive, identity-on-objects functor $(-)^*: \cC^\op \to \cC$. A morphism $u:A \to B$ in a $*$-category is \emph{unitary} if $u^* u = \id_A$ and $u u^* = \id_B$. An endomorphism $p:A\to A$ is \emph{positive} if there is some morphism $g:A\to A$ such that $p= g^* g$. A \emph{$*$-functor} $F:\cC \to \cD$ between $*$-categories is a linear functor $F$ such that $F(f^*) = F(f)^*$ for all morphisms $f$. The category \ignore{$\mathrm{Func}^\dagger(\cC, \cD)$ }of $*$-functors and natural transformations between two $*$-categories is itself a $*$-category: For a natural transformation $\eta:F \to G$, we define $\eta^*: G\to F$ to be the natural transformation with components $(\eta^*)_A:= \eta_A^*$. In particular, this gives rise to notions of \emph{positive} and \emph{unitary natural transformations}. A \emph{$*$-equivalence} $F:\cC \to \cD$ between $*$-categories is a $*$-functor $F$ such that there exists a $*$-functor $G:\cD \to \cC$ for which $F\circ G$ and $G\circ F$ are unitarily naturally isomorphic to the respective identity functors. 

A \emph{$\Cs$-category}~\cite{WStar} is a $*$-category fulfilling the following two additional properties:
\begin{itemize}
\item[a)] For every morphism $f:a\to b$, there exists a morphism $g:a\to a$ such that $f^*f = g^*g$;
\item[b)] The following expression defines a complete norm on the Hom-spaces
\[\|f\|=\sup \{|\lambda|~|~\lambda \in \mathbb{C}, \text{ s.t.} f^*f-\lambda\cdot \id \text{ is not invertible}\}
\] such that $\|fg\| \leq \|f\| \|g\|$ and $\|f^*f\| = \|f\|^2$.
\end{itemize}A \emph{$\Ws$-category}~\cite{WStar} is a $\Cs$-category with the additional property that every Banach space $\Hom(a,b)$ has a predual. 

We emphasize that being a $\Cs$-category, or a $\Ws$-category, is a property of a $*$-category, and not additional structure. 
The standard example of a $\Ws$-category is the category $\Hilb$ of finite-dimensional Hilbert spaces and linear maps. A natural transformation $\eta:F\To G$ between $*$-functors $F,G:\cC\to \cD$ between $\Cs$-categories is \emph{bounded} if $\sup_{c \in \mathrm{ob} \cC}\|\eta_c\| < \infty$. The category of $*$-functors and bounded natural transformations between two $\Cs$-categories again forms a $\Cs$-category. 

An object $x$ in a $\Cs$-category is \emph{simple} if $\End(x)\iso \mathbb{C}$. A $\Cs$-category is \emph{semisimple} if it has finite direct sums and every object is a finite direct sum of simple objects. In particular, a semisimple $\Cs$-category has finite-dimensional morphism spaces (and is hence a $\Ws$-category). 
Equivalently, a semisimple $\Cs$-category is a $*$-category that is equivalent, as a $*$-category, to a direct sum of copies of $\Hilb$.

\subsection{$\Cs$-tensor categories, $\Ws$-tensor categories and unitary fusion categories}

A \emph{$*$-monoidal category} is a linear monoidal category equipped with a $*$-structure such that all monoidal coherence isomorphisms are unitary and such that $(f\otimes g)^*= f^*\otimes g^*$ for all morphisms $f,g$. The latter condition may equivalently be expressed as stating that the tensor product functor $- \otimes -: \cC \times \cC \to \cC$ is a $*$-functor. A $\Cs$-tensor category (or a $\Ws$-tensor category) is a $*$-monoidal category whose underlying $*$-category is a $\Cs$-category (or a $\Ws$-category).

A monoidal category is \emph{rigid} if every object has a right and a left dual. A \emph{unitary multifusion category} is a rigid semisimple $\Cs$-tensor category with a finite number of simple objects. A \emph{unitary fusion category} is a unitary multifusion category with simple monoidal unit.

Recall~\cite[Def 2.4.1]{EGNO} that a \emph{monoidal functor} $\FF:\cC\to \cD$ between monoidal categories may be defined as a pair $(F,f)$ of a functor $F:\cC \to \cD$ with the property that $F(I_{\cC})$ is isomorphic to $I_{\cD}$, together with a natural isomorphism $f_{c,c'}: F(c\otimes c') \to F(c) \otimes F(c')$ fulfilling the usual coherence condition. In the following, we always denote a monoidal functor by a blackboard-bold letter $\FF$ and its underlying functor $F$ and natural isomorphism $f$ by the corresponding upper and lower case letter. A \emph{monoidal equivalence} is a monoidal functor $\FF=(F,f): \cC \to \cD$ for which $F:\cC \to \cD$ is an equivalence of categories.

A \emph{monoidal natural isomorphism} $\eta: \FF \To \GG$ between monoidal equivalences is a natural isomorphism $\eta: F \To G$ fulfilling $g_{c,c'} \eta_{c\otimes c'} = (\eta_c\otimes \eta_{c'}) f_{c,c'}$.

A \emph{unitary monoidal equivalence} between $*$-monoidal categories is a monoidal equivalence $\FF=(F,f)$ whose underlying functor $F$ is a $*$-equivalence for which $F(I_{\cC})$ is unitarily natural isomorphic to $I_{\cD}$ and whose underlying natural isomorphism $f$ is unitary.

A \emph{unitary braided fusion category} is a unitary fusion category equipped with a unitary braiding. A \emph{unitary braided monoidal equivalence} between unitary braided fusion categories is a unitary monoidal equivalence $\FF=(F,f):\cA \to \cB$ such that $ f_{a',a}F(\sigma^\cA_{a,a'})=\sigma^\cB_{F(a), F(a')}f_{a,a'}$, where $\sigma^\cA_{a,a'}:a\otimes a' \to a'\otimes a$ and $\sigma^\cB_{b,b'}:b \otimes b' \to b' \otimes b$ denote the braidings of $\cA$ and $\cB$, respectively.

\subsection{Unitary module categories}
A \emph{module category} ${}_{\cC}\cM = (\cM, {-\lact-},\mu^\cM)$ over a monoidal category $\cC$ is a category $\cM$ equipped with a functor $-\lact -: \cC \times \cM \to \cM$ with the property that ${I_{\cC}\lact-}:\cM \to \cM$ is an autoequivalence, and a natural isomorphism $\mu^\cM_{c,c', m}: (c\otimes c') \lact m \to c\lact (c' \lact m)$ fulfilling the usual coherence condition~\cite[Def 7.1.1]{EGNO}.
A \emph{module functor} $\FF =(F,f): {}_{\cC}\cM \to {}_{\cC}\cN$ between module categories is a pair of a functor $F: \cM \to \cN$ and a natural isomorphism $f_{c,m} : F(c\lact m) \to c \lact F(m)$
fulfilling the usual coherence conditions~\cite[Def 7.2.1]{EGNO}. A \emph{module equivalence} is a module functor $\FF=(F,f):{}_{\cC}\cM \to {}_{\cC} \cN$ for which $F:\cM \to \cN$ is an equivalence of categories. 

A \emph{unitary module category} ${}_{\cC}\cM =(\cM, {-\lact -}, \mu^\cM)$ over a unitary fusion category $\cC$ is a module category for which $\cM$ is a semisimple $\Cs$-category with a finite number of simple objects, for which $-\lact -: \cC \times \cM \to \cM$ is a $*$-functor with the property that ${I_{\cC} \lact -}: \cM \to \cM$ is a $*$-equivalence\footnote{By Proposition~\ref{prop:replacefunctor}, this condition is automatically satisfied.}, and for which the natural isomorphism $\mu^\cM_{c,c',m}$ is unitary. 
A \emph{unitary module equivalence} is a module equivalence $\FF=(F,f):{}_{\cC}\cM \to {}_{\cC}\cN$ between unitary module categories for which $f$ is unitary and $F$ is a $*$-equivalence.

\subsection{Polar decomposition}
Every $\Ws$-category admits a notion of \emph{polar decomposition} \cite[Cor 2.7]{WStar}. Here, we will only make use of the following special case: Every positive morphism $p:A\to A$ in a $\Ws$-category has a unique positive square root; that is, there exists a unique positive morphism $\sqrt{p}:A\to A$ such that $\sqrt{p}^2 = p$. 

This immediately gives rise to the following \emph{polar decomposition} of invertible morphisms in a $\Ws$-category:  If $f:a \to b$ is an invertible morphism, then $f$ is the composite of the unitary morphism $u = \sqrt{(f f^*)}^{-1}f : a\to b$ followed by the positive morphism $p:= \sqrt{f f^*}: b\to b$. We will make frequent use of the following direct consequence of polar decomposition.\looseness=-2
\begin{prop}\label{prop:trick} Let $\cC$ be a $\Ws$-category. Let $v$ and $w$ be unitary morphisms and $x$ and $y$ be morphisms in $\cC$ fulfilling $xv = w y$ (where the types of $v,w, x,y$ are such that this equation makes sense). Then, $|x| w =w |y| $, where $|x|$ and $|y|$ are the unique positive square roots of $x x^*$ and $ yy^*$.
\end{prop}
\begin{proof}Taking the adjoint of $xv = wy$ and using unitarity of $v$ and $w$, it follows that $x^* w = v y^*$ and in particular that $xx^* w = w y y^*$ or equivalently $|x|^2 = w |y|^2 w^*$. By uniqueness of the positive square root of positive morphisms, it follows that $|x| = w|y| w^*$.
\end{proof}

Even though the $*$-category of $*$-functors and (possibly unbounded) natural transformations between two $\Ws$-categories is not a $\Ws$-category, we can nevertheless polar decompose unbounded natural transformations.

\begin{prop}\label{prop:naturaltrafopolar} A natural endomorphism $\eta:F\To F$ of a $*$-functor between $\Cs$-categories is positive if and only if it is componentwise positive. Such a positive natural transformation admits a unique positive square root with components $(\sqrt{\eta})_c = \sqrt{\eta_c}$.
\end{prop}
\begin{proof}By definition, a natural transformation $\eta:F \To F$ is positive if there is a natural transformation $\epsilon: F\To F$ such that $\eta= \epsilon^* \epsilon$. In particular, every positive natural transformation is componentwise so. Conversely, suppose that $\eta:F\To F$ is componentwise positive. Then, for every object $c$ of $\cC$, the morphism $\eta_c:F(c) \to F(c)$ admits a unique positive square root $\epsilon_c:= \sqrt{\eta_c}:F(c) \to F(c)$. Note that given a morphism $\alpha:a\to b$ and positive morphisms $p:a\to a$ and $q:b\to b$ in a $\Cs$-category such that $q\alpha  = \alpha p$, then $\sqrt{q} \alpha = \alpha \sqrt{p}$. Using the linking algebras $M(a,b)$ developed in~\cite{WStar}, this follows for example from the analogous statement for $\Cs$-algebras. In particular, the morphisms $\epsilon_c:F(c) \to F(c)$ assemble into a natural transformation. 
\end{proof}

\begin{corollary}\label{cor:polarnatural}
Every (possibly unbounded) natural isomorphism $\eta:F\To G$ of $*$-functors between $\Ws$-categories admits a polar decomposition into a unitary (and hence bounded) natural transformation $u=\sqrt{\eta\eta^*}^{-1} \eta: F\To G$ followed by a positive (and possibly unbounded) natural transformation $p=\sqrt{\eta\eta^*}:G\To G$. 
\end{corollary}
\begin{proof}By definition, $\eta\eta^*$ is positive  and invertible and hence, following Proposition~\ref{prop:naturaltrafopolar}, admits a unique positive invertible square root $p=\sqrt{\eta \eta^*}$. A direct computation shows that $u=\sqrt{\eta\eta^*}^{-1} \eta$ is unitary.
\end{proof}

\section{On the uniqueness of unitary structure}\label{sec:main}
\subsection{Categorified polar decomposition}\label{sec:theorem:functor}
In this section, we prove a polar decomposition theorem for monoidal equivalence between $\Ws$-categories.

\begin{definition}\label{def:positivemonoidal}
Let $\cD$ be a $\Cs$-tensor category. We say that a monoidal auto-equivalence $\cD \to \cD$ is \emph{positive} if it is monoidally naturally isomorphic to a monoidal auto-equivalence $(\id_{\cD}, p)$ whose underlying functor is the identity functor and whose coherence natural isomorphism $p_{d,d'}: d\otimes d' \To d\otimes d'$ is positive. 
\end{definition}
\begin{theorem}\label{thm:functorstep1} Let $\FF=(F,f):\cC\to \cD$ be a monoidal equivalence between $\Ws$-categories whose underlying functor $F$ is naturally isomorphic to a $*$-equivalence. Then, $\FF$ is monoidally naturally isomorphic to the composite of a unitary monoidal equivalence $\cC \to \cD$ followed by a positive monoidal auto-equivalence $\cD \to \cD$.
\end{theorem}
\begin{proof} We may assume that $\FF=(F,f):\cC \to \cD$ is a monoidal equivalence whose underlying functor $F$ is a $*$-equivalence. Polar decomposition (cf. Corollary~\ref{cor:polarnatural}) of the natural isomorphism $f_{c,c'}: F(c\otimes c') \to F(c) \otimes F(c')$ results in a unitary natural isomorphism $u_{c,c'}: F(c\otimes c') \to F(c) \otimes F(c')$ followed by a positive natural isomorphism $q_{c,c'}:F(c) \otimes F(c') \to F(c) \otimes F(c')$. Since $F$ is a $*$-equivalence, there is a positive natural isomorphism $p_{d,d'}: d\otimes d' \to d\otimes d'$ such that $q_{c,c'} = p_{F(c), F(c')}$. (Note that this is the only step in the proof where we use the crucial fact that $\FF$ is an equivalence rather than an arbitrary monoidal functor.) 

We first show that $(\id_{\cD}, p)$ is a monoidal equivalence. 
In terms of the composites
\begin{align}\label{eq:fR} f^R_{c,c',c''}&: = F(c\otimes (c'\otimes c'')) \!\to[f_{c,c'\otimes c''}] \!F(c) \otimes F(c'\otimes c'') \!\to[F(c) \otimes f_{c',c''}]\! F(c) \otimes (F(c') \otimes F(c''))\hspace{-0.2cm}
\\\nonumber
f_{c,c',c''}^L&:= F((c\otimes c') \otimes c'') \!\to[f_{c\otimes c', c''}] \!F(c\otimes c') \otimes F(c'') \!\to[f_{c,c'}\otimes F(c'')] \!(F(c)\otimes F(c')) \otimes F(c'')\hspace{-0.2cm}
\end{align}
the coherence equation for $\FF$ can be written as
\begin{equation}\label{eq:coherence}f_{c,c',c''}^R F(\alpha^\cC_{c,c',c''}) = \alpha^\cD_{F(c), F(c'), F(c'')} f_{c,c',c''}^L
\end{equation}
where $\alpha^\cC_{c,c',c''}: (c\otimes c') \otimes c'' \to c\otimes (c' \otimes c'')$ and $\alpha_{d,d',d''}^\cD:  (d\otimes d') \otimes d'' \to d\otimes (d' \otimes d'')$ are the associator unitary natural isomorphisms of $\cC$ and $\cD$, respectively.

It follows from naturality of $p$ that 
\[p_{F(c), F(c')\otimes F(c'')} (F(c) \otimes u_{c', c''} ) = (F(c) \otimes u_{c',c''}) p_{F(c), F(c'\otimes c'')}
\]
and hence that $f^R_{c,c',c''}$ can be expressed as the composite\begin{equation}\label{eq:rewrittengraphical}f^R_{c,c',c''} = p^R_{F(c),F(c'),F(c'')} u^R_{c,c',c''}\end{equation} where $u^R$ and $p^R$ are defined as follows:
\begin{align}\label{eq:pR}
u^R_{c,c',c''}&:= F(c\otimes (c'\otimes c'')) \to[u_{c,c'\otimes c''}] \!F(c)\! \otimes\! F(c'\otimes c'') \to[F(c) \otimes u_{c',c''}] \!F(c) \!\otimes\! (F(c') \!\otimes \!F(c'')) \hspace{-0.1cm}
\\\nonumber
p^R_{d,d',d''}&:= d \otimes (d'\otimes d'')\to[ p_{d, d' \otimes d''}] d \otimes (d' \otimes d'')\to [d\otimes p_{d' , d''}]d\otimes (d'\otimes d'') 
\end{align}
Similarly, it follows from positivity of $p $ and naturality of its square root $\sqrt{p}$ that the composite $p^R$ is equal to the composite $g^* g$, where $g$ is defined as follows
\[g:=d\otimes (d' \otimes d'') \to[\sqrt{p}_{d,d'\otimes d''}] d\otimes (d' \otimes d'') \to[d\otimes \sqrt{p}_{d', d''}] d\otimes (d' \otimes d'')
\] and hence is positive. 
In particular, $p_{F(c),F(c'),F(c'')}^R$ is the unique positive square root of $f^R_{c,c',c''}(f^R_{c,c',c''})^*$. Defining $u^L$ and $p^L$ analogously and applying Proposition~\ref{prop:trick} to equation~\eqref{eq:coherence} implies that \[p^R_{F(c),F(c'),F(c'')}\alpha^\cD_{F(c), F(c'), F(c'')}  =\alpha^\cD_{F(c), F(c'), F(c'')} p^L_{F(c),F(c'),F(c'')}.
\]
Since $F$ is a $*$-equivalence, this is equivalent to the coherence equation for $(\id_{\cD}, p)$. 

By definition $(F, u)$ can be written as the composite of the monoidal equivalence $\FF: \cC \to \cD$ followed by $(\id_{\cD}, p^{-1}): \cD \to \cD$ and is hence also a monoidal equivalence.
\end{proof}

It follows from the following well-known observation~\cite[Rem 2.7]{CompleteUnitary}\cite[Rem 3.4]{PenneysUnitary}, that the assumptions of Theorem~\ref{thm:functorstep1} are automatically satisfied for semisimple $\Cs$-categories. 

\begin{prop} \label{prop:replacefunctor}Every linear functor between semisimple $\Cs$-categories is naturally isomorphic to a $*$-functor, which is uniquely determined up to unitary natural isomorphism. Moreover, a $*$-functor is an equivalence if and only if it is a $*$-equivalence.
\end{prop}
\begin{proof}  Let $\mathrm{Irr}(\cC)$ and $\mathrm{Irr}(\cD)$ be sets of representing simple objects of $\cC$ and $\cD$, respectively. Any linear functor $F:\cC \to \cD$ is completely determined up to natural isomorphism by the (finite) dimensions of the vector spaces $\Hom_{\cD}(d,Fc)$ for $c\in \mathrm{Irr}(\cC)$ and $d\in \mathrm{Irr}(\cD)$ and is in particular naturally isomorphic to the $*$-functor $\widetilde{F}$ determined by the same data.
It follows from polar decomposition of natural transformations (Corollary~\ref{cor:polarnatural}) that the $*$-functor $\widetilde{F}$ is uniquely determined up to unitary natural isomorphism. Together with polar decomposition, this implies that every $*$-functor which is an equivalence is a $*$-equivalence.\looseness=-2
\end{proof}

Combining Theorem~\ref{thm:functorstep1} with Proposition~\ref{prop:replacefunctor}, we obtain the following polar decomposition theorem for semisimple $\Cs$-categories.
\begin{cor}\label{cor:polarsemisimple}Every monoidal equivalence between semisimple $\Cs$-categories $\cC\to \cD$ is monoidally naturally isomorphic to the composite of a unitary monoidal equivalence followed by a positive monoidal auto-equivalence. \\In particular, every semisimple $\Cs$-tensor category has a unique $\Cs$-structure, that is, if $\cC$ and $\cD$ are semisimple $\Cs$-tensor categories and if there is a monoidal equivalence between $\cC$ and $\cD$, then there is a unitary monoidal equivalence.
\end{cor}

\begin{remark}\label{rem:positiveauto} In contrast to unitary fusion categories (Proposition~\ref{prop:functorstep2}), semisimple $\Cs$-tensor categories with infinitely many simple objects can admit non-trivial positive monoidal auto-equivalences. For example, for a (discrete, possibly infinite) group $G$ let $\Hilb_G$ be the category of $G$-graded finite-dimensional Hilbert spaces. The group of monoidal natural isomorphism classes of positive monoidal auto-equivalences of $\Hilb_G$ is isomorphic to $H^2(G, \mathbb{R}_{> 0})$ and may therefore be non-trivial. In particular, the first part of Theorem~\ref{thm:functor} is not true for such categories. 
\end{remark}

\subsection{Unitary fusion categories}For the rest of the paper, we restrict attention to unitary fusion categories. Theorem~\ref{thm:functor} follows from the fact that every positive monoidal auto-equivalence of a unitary fusion category is trivial.

\begin{prop} \label{prop:functorstep2}Every positive monoidal auto-equivalence $\cC \to \cC$ of a unitary fusion category $\cC$ is monoidally naturally isomorphic to the identity monoidal equivalence. 
\end{prop}
\begin{proof} 
By definition, every positive monoidal equivalence is naturally isomorphic to a monoidal equivalence $\PP = (\id_{\cC}, p)$ where $p_{c,c'}: c\otimes c' \to c\otimes c'$ is positive. It follows from~\cite[Thm 4.15]{ENOHomotopy} that the group $\Eq(\cC)$ of monoidal autoequivalences of $\cC$ up to natural isomorphisms is finite. In particular, there is a natural number $n$ such that $\PP^n = (\id_{\cC}, p^n)$ is naturally isomorphic to the identity monoidal equivalence. In other words, there is a natural isomorphism $\eta: \id_{\cC} \To \id_{\cC}$ such that $p_{c,c'}^n = \eta_{c\otimes c'}^{-1} (\eta_{c}\otimes \eta_{c'})$.
Since $p$ is self-adjoint, it follows from repeated use of naturality of $\eta$ that \[p^{2n}_{c,c'} = (\eta_{c}^\dagger \otimes \eta_{c'}^\dagger) \left(\eta_{c\otimes c'}^{-1}\right)^\dagger \eta^{-1}_{c\otimes c'} (\eta_c\otimes \eta_{c'}) = (\eta_{c\otimes c'}\eta_{c\otimes c'}^\dagger)^{-1} (\eta_c^\dagger\eta_c \otimes \eta_{c'}^\dagger\eta_{c'}) \]
\[=  (\eta_{c\otimes c'}^\dagger\eta_{c\otimes c'})^{-1} (\eta_c^\dagger\eta_c \otimes \eta_{c'}^\dagger\eta_{c'}) = \mu_{c\otimes c'}^{-1} (\mu_c\otimes \mu_{c'}),
\] 
where $\mu = \eta^\dagger \eta$ is positive. Denote the unique positive $2n$th root of $\mu$ by $\epsilon: \id_{\cC} \To \id_{\cC}$ and note that $\epsilon$ is invertible. It follows from uniqueness of $2n$th roots that $p_{c,c'} = \epsilon_{c\otimes c'}^{-1}(\epsilon_c \otimes \epsilon_{c'})$
, and hence that $(\id_{\cC}, p)$ is monoidally naturally isomorphic to $(\id_{\cC}, \id_{-\otimes-})$. 
\end{proof}

Theorem~\ref{thm:functor} follows as a direct consequence of Corollary~\ref{cor:polarsemisimple} and~\ref{prop:functorstep2}.

\subsection{Unitary monoidal natural isomorphisms}\label{sec:theorem:trafo}

In this section, we show that every monoidal natural isomorphism between unitary monoidal equivalences is automatically unitary. The proof is completely analogous to the proof of Theorem~\ref{thm:functor}: First, we decompose the monoidal natural isomorphism into a unitary monoidal natural isomorphism followed by a positive monoidal natural isomorphism, and then we use a finiteness argument to show that there is no non-trivial positive monoidal natural isomorphism.

\begin{prop}\label{prop:trafostep1} Let $\FF,\GG: \cC \to \cD$ be unitary monoidal equivalences between unitary fusion categories. Then, every monoidal natural isomorphism $\eta: \FF \To \GG$ is the composite of a unitary monoidal natural isomorphism $u:\FF \To \GG$ followed by a positive monoidal natural automorphism $p:\GG \To \GG$.
\end{prop}
\begin{proof} The natural isomorphism $\eta: F \To G$ factors as a unitary natural isomorphism $u:F \To G$ followed by a positive natural isomorphism $p: G \To G$. 
Applying Proposition~\ref{prop:trick} to the coherence equation
\[(\eta_c \otimes \eta_{c'} )f_{c,c'}=g_{c,c'} \eta_{c\otimes c'} 
\]
proves
\[(p_c \otimes p_{c'}) g_{c,c'} = g_{c,c'} p_{c\otimes c'}
\]
and hence that $p : G\To G$ is a monoidal natural automorphism. 

By definition, the unitary natural isomorphism $u:F\To G$ can be written as the composite $u= p^{-1} \eta: \FF \To \GG$ of monoidal natural isomorphisms and is therefore also a monoidal natural isomorphism.
\end{proof}

\begin{prop}\label{prop:trafostep2}Let $\FF = (F,f):\cC \to \cD$ be a unitary monoidal equivalence between unitary fusion categories. Every positive monoidal natural automorphism $\eta: \FF \To \FF$ equals the identity. 
\end{prop}
\begin{proof} Since $F$ is a $*$-equivalence, every positive natural isomorphism $\eta: F\To F$ is of the form $F\circ p$, where $p: \id_{\cC} \To \id_{C}$ is a positive natural isomorphism. It follows from naturality of $f$ that the monoidality equation for $\eta = F\circ p$ may be rewritten as follows:
\[ f_{c,c'} F(p_{c\otimes c'}) =  (F(p_c) \otimes F(p_{c'}))f_{c,c'} =f_{c,c'} F(p_{c} \otimes p_{c'})
\]Invertibility of $f$ and the fact that $F$ is an equivalence imply that $p_{c\otimes c'} = p_{c} \otimes p_{c'}$\ignore{$p\circ m_{\cC} = m_{\cC} \circ (p \boxtimes p)$}, and hence that $p$ is a positive monoidal natural automorphism of the identity unitary monoidal equivalence $\cC \to \cC$. It is shown in~\cite[Lem 3.19]{PenneysUnitary} that the group $\mathrm{Aut}^+(\id_{\cC})$ of positive monoidal natural automorphisms of the identity is isomorphic to the group of group homomorphisms $\Hom(\cU_{\cC}, \RR_{>0})$ from the \emph{universal grading group}~\cite[Def 4.14.2]{EGNO} of $\cC$ to $\RR_{>0}$. Since $\cC$ is a unitary fusion category, the group $\cU_{\cC}$ is finite and $\Hom(\cU_{\cC}, \RR_{>0})$ is trivial.
\end{proof}

Theorem~\ref{thm:trafo} is a direct consequence of Propositions~\ref{prop:trafostep1} and~\ref{prop:trafostep2}.

\begin{remark} If $\cC$ is a unitary multifusion category, the group $\Hom(\cU_{\cC}, \RR_{>0})$ is non-trivial if $\cC$ has more than one summand. In this case, there are non-trivial positive monoidal natural isomorphisms between unitary monoidal equivalences. In particular, Proposition~\ref{prop:trafostep2} and therefore also Theorem~\ref{thm:trafo} and Corollary~\ref{cor:Eq} do not hold for general multifusion categories.
\end{remark}

\subsection{Unitary braided fusion categories}\label{sec:theorem:braiding}
Turning our attention to braided monoidal categories, we now show that Theorem~\ref{thm:braiding} immediately follows from the proof of Theorem~\ref{thm:functor}.
\begin{proof}[Proof of Theorem~\ref{thm:braiding}] By Proposition~\ref{prop:replacefunctor}, every braided monoidal equivalence between unitary braided fusion categories is naturally isomorphic to a braided monoidal equivalence $\FF=(F,f): \cA \to \cB$ whose underlying functor $F$ is a dagger equivalence. Following the proof of Theorem~\ref{thm:functorstep1}, we factor $\FF$ into a unitary monoidal equivalence $(F,u): \cA \to \cB$ followed by a positive monoidal auto-equivalence $(\id_{\cB}, p): \cB \to \cB$, and show that $(\id_{\cB}, p)$ is in fact a braided monoidal auto-equivalence. Indeed, applying Proposition~\ref{prop:trick} to the compatibility condition
\[f_{a',a} F(\sigma^\cA_{a,a'}) = \sigma^\cB_{F(a), F(a')} f_{a,a'}
\]
shows that 
\[p_{F(a'), F(a)} \sigma^\cB_{F(a), F(a')} = \sigma^\cB_{F(a), F(a')} p_{F(a), F(a')}
\]
and hence that $(\id_{\cB}, p)$ is a braided monoidal auto-equivalence. By definition, $(F,u)$ is the composite of $\FF: \cA \to \cB$ followed by $(\id_{\cB}, p^{-1}):\cB \to \cB$ and is hence also a braided monoidal equivalence. The theorem then follows from Proposition~\ref{prop:functorstep2} and the fact that $(\id_{\cB}, p)$ is monoidally naturally isomorphic to the identity.
\end{proof}
\subsection*{Unitary module categories}\label{sec:theorem:modfunctor}
The proof of Theorem~\ref{thm:modfunctor} is completely analogous to the proof of Theorem~\ref{thm:functor}; we will first show that every module equivalence factors into a unitary module equivalence followed by a positive module equivalence, and then show that every positive module equivalence is trivial.
The proofs of Propositions~\ref{prop:modfunctorstep1} and~\ref{prop:modfunctorstep2} are completely analogous to the proofs of Theorem~\ref{thm:functorstep1} and Proposition~\ref{prop:functorstep2}. 
For the reader's convenience, we spell them out again, following the wording of Theorem~\ref{thm:functorstep1} and Proposition~\ref{prop:functorstep2} as closely as possible.

\begin{definition} Let ${}_{\cC}\cM$ be a unitary module category over a unitary fusion category. We say that a module auto-equivalence ${}_{\cC}\cM \to {}_{\cC}\cM$ is \emph{positive} if it is naturally isomorphic, as a module functor, to a module functor $(\id_{\cM}, p)$ whose underlying functor is the identity functor and whose coherence natural isomorphism $p_{c, m} : c\lact m \to c \lact m$ is positive. 
\end{definition}
\begin{prop} \label{prop:modfunctorstep1}Every module equivalence $F: {}_{\cC}\cM \to{}_{\cC}\cN$ between unitary module categories over a unitary fusion category $\cC$ is naturally isomorphic, as a module functor, to the composite of a unitary module equivalence followed by a positive module auto-equivalence.
\end{prop}
\begin{proof}

By Proposition~\ref{prop:replacefunctor}, every module equivalence ${}_{\cC}\cM \to {}_{\cC}\cN$ is naturally isomorphic to a module equivalence $(F, f)$  whose underlying functor $F$ is a $*$-functor. Polar decomposition of the natural isomorphism $f_{c,m}: F(c \lact m) \to c\lact F(m)$ results in a unitary natural isomorphism $u_{c,m}: F(c\lact m) \to c\lact F(m)$ followed by a positive natural isomorphism $q_{c,m}: c\lact F(m)\to c \lact F(m)$. Since $F$ is a $*$-equivalence, there is a positive natural isomorphism $p_{c, n}: c\lact n \to c \lact n$ such that $q_{c,m} = p_{c, F(m)}$.  

We first show that $(\id_{\cN}, p)$ is a module equivalence. In terms of the composite
\[f^R_{c,c',m}:=F(c\lact(c'\lact m))\to[f_{c,c'\lact m}] c\lact F(c' \lact m) \to[c\lact f_{c', m}] c\lact (c' \lact F(m))
\]
the coherence equation~\cite[Eq (7.6)]{EGNO} for $\FF$ can be written as follows:
\begin{equation}\label{eq:modulecoherence}f^R_{c,c',m} F(\mu^{\cM}_{c,c', m}) = \mu^\cN_{c,c',F(m)}f_{c\otimes c', m}
\end{equation}
where $\mu^\cM_{c,c',m}:( c\otimes c')\lact m \to c\lact (c' \lact m)$ and $\mu^\cN_{c,c',n}:(c\otimes c') \lact n \to c\lact(c' \lact n)$ are the coherence isomorphisms of $\cM$ and $\cN$, respectively.

As in the proof of Theorem~\ref{thm:functorstep1}, it follows from naturality of $p$ that \ignore{
\[p_{c, c'\lact F(m)} (c \lact u_{c', m} ) = (c \lact  u_{c',m}) p_{c, F(c'\lact m)}
\]and hence that }$f^R_{c,c',m}$ can be expressed as the composite 
\[f^R_{c,c',m} = p^R_{c,c', F(m)} u^R_{c,c', m}
\]
where $u^R$ and $p^R$ are defined as follows:
\begin{align}\nonumber
u^R_{c,c',c''}&:= F(c\lact (c'\lact m)) \to[u_{c,c'\lact m}] c \lact F(c'\lact m) \to[c\lact u_{c',m}] c\lact (c' \lact F(m))
\\\nonumber
p^R_{c,c',n}&:= c \lact (c'\lact n)\to[ p_{c, c' \lact n}] c \lact (c' \lact n)\to [c\lact p_{c' , n}]c\lact (c'\lact n) 
\end{align}
Again as in the proof of Theorem~\ref{thm:functorstep1}, it follows from positivity of $p $ and naturality of its square root $\sqrt{p}$ that the composite $p^R$ is positive. 
In particular, $p_{c,c',F(m)}^R$ is the unique positive square root of $f^R_{c,c',m}(f^R_{c,c',m})^\dagger$. Applying Proposition~\ref{prop:trick} to equation~\eqref{eq:modulecoherence} implies that \[p^R_{c,c',F(m)}\mu^\cN_{c, c', F(m)}  =\mu^\cN_{c, c', F(m)} p_{c\otimes c',F(m)}.
\]
Since $F$ is a $*$-equivalence, this is equivalent to the coherence equation for $(\id_{\cN}, p)$. Compatibility with the unitors~\cite[Eq (7.7)]{EGNO} $l_m: I_{\cC} \lact m \iso m$ follows analogously. 

By definition $(F, u)$ can be written as the composite of the module equivalence $\FF: {}_\cC\cM \to {}_{\cC}\cN$ followed by $(\id_{\cN}, p^{-1}): {}_{\cC}\cN \to {}_{\cC}\cN$ and is hence also a module equivalence.
\end{proof}

\begin{prop} \label{prop:modfunctorstep2}Every positive module auto-equivalence ${}_\cC\cM \to {}_\cC\cM$ of a unitary module category ${}_\cC\cM$ is naturally isomorphic, as a module functor, to the identity module equivalence. 
\end{prop}
\begin{proof} 
By definition, every positive module equivalence is naturally isomorphic to a module equivalence $\PP = (\id_{\cM}, p)$ where $p_{c,m}: c\lact m \to c\lact m$ is positive. The group $\Aut_{\cC}(\cM)$ of module auto-equivalences ${}_{\cC}\cM \to {}_{\cC}\cM$ up to module natural isomorphisms is finite. (This can for example be seen by noting that the monoidal category $\End_{\cC}(\cM)$ of module endofunctors is a multifusion category, and that $\Aut_{\cC}(\cM)$ is the group of invertible objects in this multifusion category.) In particular, there is a natural number $n$ such that $\PP^n = (\id_{\cM}, p^n)$ is naturally isomorphic to the identity module functor. In other words, there is a natural isomorphism $\eta: \id_{\cM} \To \id_{\cM}$ such that $p_{c,m}^n =\eta_{c\lact m}^{-1} ( c \lact \eta_m)$. 
Since $p$ is self-adjoint, it follows from repeated use of naturality of $\eta$ that $p^{2n}_{c,m} = \mu_{c\lact m}^{-1}(c\lact \mu_m)$
where $\mu = \eta^\dagger \eta$ is positive. Denote the unique positive $2n$th root of $\mu$ by $\epsilon: \id_{\cM} \To \id_{\cM}$ and note that $\epsilon$ is invertible. It follows from uniqueness of $2n$th roots that $p_{c,m} = \epsilon_{c\lact m}^{-1}(c \lact \epsilon_{m})$
, and hence that $(\id_{\cM}, p)$ is naturally isomorphic, as a module functor to $(\id_{\cM}, \id_{-\lact-})$. 
\end{proof}
Theorem~\ref{thm:modfunctor} is a direct consequence of Propositions~\ref{prop:modfunctorstep1} and~\ref{prop:modfunctorstep2}.

\section{A monoidal $2$-categorical perspective}\label{sec:2Hilb}
\def\picgap{\hspace{0.3cm}} Most of the proofs in this paper were developed using the graphical calculus of the monoidal dagger $2$-category $\tHilb$ of finitely semisimple $\Cs$-categories, $*$-functors and natural transformations. (This monoidal $2$-category is equivalent to Baez's $2$-category of `finite-dimensional $2$-Hilbert spaces'~\cite{Baez} and is a unitary version of Kapranov and Voevodsky's $2$-category of `finite-dimensional $2$-vector spaces~\cite{kv}.) Many of the more subtle applications of naturality in the proofs of Theorem~\ref{thm:functorstep1} and Proposition~\ref{prop:functorstep2} become transparent once expressed in this graphical calculus. To give a flavour of such arguments, we sketch the relevant parts of the proof of Theorem~\ref{thm:functorstep1} (in the case that both $\cC$ and $\cD$ are unitary fusion categories) in this language. The following section is purely expositional and not relevant to the mathematical developments of Sections~\ref{sec:prelim} and~\ref{sec:main}.

Proposition~\ref{prop:replacefunctor} shows that all structural data in question --- the monoidal category $\cC$ with tensor product $m_{\cC} := -\otimes -: \cC \boxtimes \cC\to \cC$ and associator $\alpha: m_{\cC} \circ (m_{\cC} \boxtimes \id_{\cC} ) \To  m_{\cC} \circ ( \id_{\cC}\boxtimes m_{\cC})$, as well as the monoidal equivalence $\FF= (F,f)$ with underlying $*$-functor $F:\cC \to \cD$ and natural isomorphism $f:F \circ m_{\cC} \To m_{\cD} \circ (F\boxtimes F)$ --- are given by objects, $1$- and $2$-morphisms of $\tHilb$. As a monoidal $2$-category, $\tHilb$ admits a graphical calculus of \emph{surface diagrams}\footnote{Strictly speaking, surface diagrams form the graphical calculus of semistrict monoidal $2$-categories, so called \emph{Gray monoids}. The coherence theorem for weak 3-categories~\cite{Gurski} justifies working with this graphical calculus even in the context of weak monoidal $2$-categories, c.f.~\cite{Guthmann}. }  in 3-space~\cite{BMS}. We draw 1-morphism composition from right to left, 2-morphism composition from bottom to top, and depict the monoidal structure by layering surfaces behind one another, with the convention that tensor product occurs from back to front: that is, in a diagram for $A \boxtimes B$, the surface labeled $A$ appears in front of the surface labeled $B$ (see~\cite[Sec 2.1.2]{FTC} for a more careful description of our conventions). 

For example, the associator $\alpha$ of a unitary fusion category is depicted as follows:
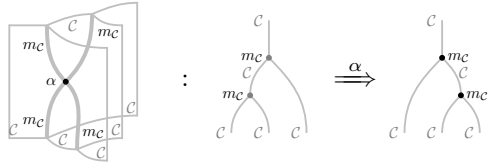
\tikzset{slicem/.style={draw = gray!50, line width = 1.5pt}}
\[
\raisebox{-0.08cm}{%
\begin{tz}[td,scale=1]
\begin{scope}[xyplane=0]
\draw[slice] (0,0) to [out=up, in=\dl] (0.5,1) to [out=up, in=\dl] (1,2) to (1,3);
\draw[slice] (1,0) to [out=up, in=\dr] (0.5,1);
\draw[slice] (2,0) to [out=up, in=\dr] (1,2);
\end{scope}
\begin{scope}[xyplane=\h, on layer=superfront]
\draw[slice] (0,0) to [out=up, in=\dl] (1,2) to (1,3);
\draw[slice] (1,0) to [out=up, in=\dl] (1.5,1) to [out=up, in=\dr] (1,2);
\draw[slice] (2,0) to [out=up, in=\dr] (1.5,1);
\end{scope}
\begin{scope}[xzplane=0]
\draw[slice,short] (0,0) to (0, \h);
\draw[slice,short] (1,0) to (1,\h);
\draw[slice,short] (2,0) to (2,\h);
\end{scope}
\begin{scope}[xzplane=3]
\draw[slice,short] (1,0) to (1,\h);
\end{scope}
\coordinate (A) at (1.5,1,0.5*\h);
\draw[slicem] (1,0.5,0) to [out=up, in=\dr] (A) to [out=\ur, in=down] (1,1.5,\h);
\draw[slicem] (2,1,0) to [out=up, in=\dl] (A) to [out=\ul, in=down] (2,1,\h);
\node[dot] at (A){};
\node[obj, above left] at (-0.12,0,-0.08) {$\cC$};
\node[obj, above left] at (-0.12,1,-0.08) {$\cC$};
\node[obj, above left] at (-0.12,2,-0.08) {$\cC$};
\node[obj, above right] at (3.15,1,-0.08){$\cC$};
\node[obj, above] at (1.4,0.6,-0.08){$\cC$};
\node[obj, below] at (1.5,1.5,\h){$\cC$};
\node[omor, above left] at (1.9,1,0.05*\h){$m_{\cC}$};
\node[omor, above right] at (1,0.5,0.15){$m_{\cC}$};
\node[omor, below left] at (1.9,1,0.9*\h){$m_{\cC}$};
\node[omor, below right] at (1,1.5,\h-0.25){$m_{\cC}$};
\node[tmor, left] at ([xshift=-0.1cm]A){$\alpha$};
\end{tz}
}
\hspace{0.2cm}:\hspace{0.2cm}
\begin{tz}[scale=0.5]
\draw[slice] (0,0) to [out=up, in=\dl] (0.5,1) to [out=up, in=\dl] (1,2) to (1,3);
\draw[slice] (1,0) to [out=up, in=\dr] (0.5,1);
\draw[slice] (2,0) to [out=up, in=\dr] (1,2);
\node[dot,gray] at (0.5,1){};
\node[dot,gray] at (1,2){};
\node[obj,left] at (0,0) {$\cC$};
\node[obj,  left] at (1,0) {$\cC$};
\node[obj, left] at (2,0) {$\cC$};
\node[obj,  left] at (1, 3) {$\cC$};
\node[obj,  left] at (0.7, 1.6) {$\cC$};
\node[omor, left] at (0.5,1) {$m_{\cC}$};
\node[omor, left] at (1,2) {$m_{\cC}$};
\end{tz}
\hspace{0.25cm}\To[~\alpha~]\hspace{0cm}
\begin{tz}[scale=0.5,xscale=-1]
\draw[slice] (0,0) to [out=up, in=\dl] (0.5,1) to [out=up, in=\dl] (1,2) to (1,3);
\draw[slice] (1,0) to [out=up, in=\dr] (0.5,1);
\draw[slice] (2,0) to [out=up, in=\dr] (1,2);
\node[dot] at (0.5,1){};
\node[dot] at (1,2){};
\node[obj,left] at (0,0) {$\cC$};
\node[obj,  left] at (1,0) {$\cC$};
\node[obj, left] at (2,0) {$\cC$};
\node[obj,  left] at (1, 3) {$\cC$};
\node[obj,  right] at (0.7, 1.6) {$\cC$};
\node[omor, right] at (0.5,1) {$m_{\cC}$};
\node[omor, right] at (1,2) {$m_{\cC}$};
\end{tz}
\]
Here, the thick gray wires denote the $1$-morphism $m_{\cC}$, and the central black node denotes the $2$-isomorphism $\alpha$. (The thin gray bounding wires simply indicate the extent of the picture.)
For clarity, we have also explicitly depicted the source and target of $\alpha$; that source and target appear in the surface diagram as the bottom and top horizontal slices, respectively.  Note that we will often omit labels on regions, wires, and nodes, when it is clear from context what those labels should be.
The coherence natural isomorphism $f$ of a monoidal equivalence $\FF= (F,f): \cC \to \cD$ is depicted as follows:
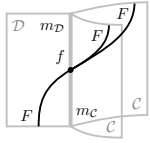
\[\begin{tz}[td,xscale=1.7]
\begin{scope}[xyplane=0]
\draw[slice] (0,0) to [out=up, in=\dl] node[pos=0.1] (tl){} (0.5,1) to (0.5,2);
\draw[slice] (1,0) to [out=up, in=\dr]node[pos=0.1] (tr){} (0.5,1);
\end{scope}
\begin{scope}[xyplane=\h, on layer=superfront]
\draw[slice] (0,0) to [out=up, in=\dl] (0.5,1) to (0.5,2);
\draw[slice] (1,0) to [out=up, in=\dr] (0.5,1);
\end{scope}
\begin{scope}[xzplane=0]
\draw[slice,short] (0,0) to (0, \h);
\draw[slice,short] (1,0) to (1,\h);
\end{scope}
\begin{scope}[xzplane=2]
\draw[slice,short] (0.5,0) to (0.5,\h);
\end{scope}
\begin{scope}[xzplane=1]
\draw[slicem,short] (0.5,0) to (0.5,\h);
\end{scope}
\coordinate (A) at (1,0.5,0.5*\h);
\draw[wire] (1.5,0.5, 0) to [out=up, in=\dl] (A) to [out=\ur, in=down] (0.2,1,\h);
\draw[wire] (A) to [out=\ur, in=down] (0.2,0,\h);
\node[dot] at (A){};
\node[above left, tmor] at (A) {$f$};
\node[obj, above left] at (-0.0,0,-0.0) {$\cC$};
\node[obj, above left] at (-0.0,1,-0.0) {$\cC$};
\node[obj, below right] at (2,0.5,\h-0.02){$\cD$};
\node[omor, above left] at (1.5,0.5,0){$F$};
\node[omor, below left] at (0.2,1,\h){$F$};
\node[omor, below left] at (0.2,0,\h){$F$};
\node[omor, above right] at (1,0.5, 0.1) {$m_{\cC}$};
\node[omor, below left] at (1,0.5, \h-0.1) {$m_{\cD}$};
\end{tz}
\]
The proof of Theorem~\ref{thm:functorstep1} begins by polarly decomposing the natural isomorphism $f$ into a unitary natural isomorphism $u:F \circ m_{\cC} \To m_{\cD} \circ (F \boxtimes F)$ followed by a positive natural isomorphism $p:m_{\cD} \To m_{\cD}$, depicted as follows:
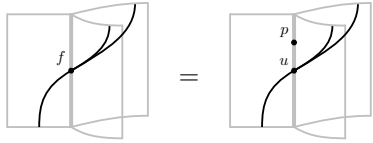
\[
\begin{tz}[td,xscale=1.7]
\begin{scope}[xyplane=0]
\draw[slice] (0,0) to [out=up, in=\dl] node[pos=0.1] (tl){} (0.5,1) to (0.5,2);
\draw[slice] (1,0) to [out=up, in=\dr]node[pos=0.1] (tr){} (0.5,1);
\end{scope}
\begin{scope}[xyplane=\h, on layer=superfront]
\draw[slice] (0,0) to [out=up, in=\dl] (0.5,1) to (0.5,2);
\draw[slice] (1,0) to [out=up, in=\dr] (0.5,1);
\end{scope}
\begin{scope}[xzplane=0]
\draw[slice,short] (0,0) to (0, \h);
\draw[slice,short] (1,0) to (1,\h);
\end{scope}
\begin{scope}[xzplane=2]
\draw[slice,short] (0.5,0) to (0.5,\h);
\end{scope}
\begin{scope}[xzplane=1]
\draw[slicem,short] (0.5,0) to (0.5,\h);
\end{scope}
\coordinate (A) at (1,0.5,0.5*\h);
\draw[wire] (1.5,0.5, 0) to [out=up, in=\dl] (A) to [out=\ur, in=down] (0.2,1,\h);
\draw[wire] (A) to [out=\ur, in=down] (0.2,0,\h);
\node[dot] at (A){};
\node[above left, tmor] at (A) {$f$};
\end{tz}
\picgap=\picgap
\begin{tz}[td,xscale=1.7]
\begin{scope}[xyplane=0]
\draw[slice] (0,0) to [out=up, in=\dl] node[pos=0.1] (tl){} (0.5,1) to (0.5,2);
\draw[slice] (1,0) to [out=up, in=\dr]node[pos=0.1] (tr){} (0.5,1);
\end{scope}
\begin{scope}[xyplane=\h, on layer=superfront]
\draw[slice] (0,0) to [out=up, in=\dl] (0.5,1) to (0.5,2);
\draw[slice] (1,0) to [out=up, in=\dr] (0.5,1);
\end{scope}
\begin{scope}[xzplane=0]
\draw[slice,short] (0,0) to (0, \h);
\draw[slice,short] (1,0) to (1,\h);
\end{scope}
\begin{scope}[xzplane=2]
\draw[slice,short] (0.5,0) to (0.5,\h);
\end{scope}
\begin{scope}[xzplane=1]
\draw[slicem,short] (0.5,0) to (0.5,\h);
\end{scope}
\coordinate (A) at (1,0.5,0.5*\h);
\draw[wire] (1.5,0.5, 0) to [out=up, in=\dl] (A) to [out=\ur, in=down] (0.2,1,\h);
\draw[wire] (A) to [out=\ur, in=down] (0.2,0,\h);
\node[dot] at (A){};
\node[dot] (B) at (1,0.5, 0.75*\h){};
\node[above left, tmor] at (A) {$u$};
\node[above left, tmor] at (B) {$p$};
\end{tz}
 \]
Next, we use naturality of $p$ to re-expresses the composite $f^R$, defined in equation~\eqref{eq:fR}, as the composite of equation~\eqref{eq:rewrittengraphical}. Graphically, this corresponds to the following simple isotopy:
 \def\top{1.5*\h}
 \def\topb{0.5*\h}
 \def\topt{0.7*\h}
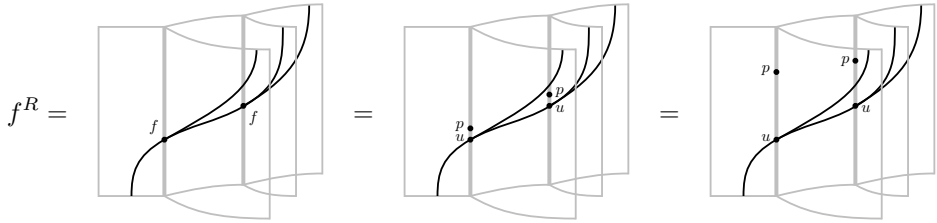
 \[f^R =\picgap \begin{tz}[td,scale=1,xscale=1.75,yscale=1]
\begin{scope}[xyplane=0]
\draw[slice] (0,0) to [out=up, in=\dl] (1,2) to (1,3);
\draw[slice] (1,0) to [out=up, in=\dl] (1.5,1) to [out=up, in=\dr] (1,2);
\draw[slice] (2,0) to [out=up, in=\dr] (1.5,1);
\end{scope}
\begin{scope}[xyplane=\top, on layer=superfront]
\draw[slice] (0,0) to [out=up, in=\dl] (1,2) to (1,3);
\draw[slice] (1,0) to [out=up, in=\dl] (1.5,1) to [out=up, in=\dr] (1,2);
\draw[slice] (2,0) to [out=up, in=\dr] (1.5,1);
\end{scope}
\begin{scope}[xzplane=0, on layer =superfront]
\draw[slice,short] (0,0) to (0, \top);
\draw[slice,short] (1,0) to (1,\top);
\draw[slice,short] (2,0) to (2,\top);
\end{scope}
\begin{scope}[xzplane=3]
\draw[slice,short] (1,0) to (1,\top);
\end{scope}
\begin{scope}[xzplane=2]
\draw[slicem,short] (1,0) to (1,\top);
\end{scope}
\begin{scope}[xzplane=1]
\draw[slicem,short] (1.5,0) to (1.5,\top);
\end{scope}
\coordinate (A) at (2,1,0.5*\h);
\node[dot] (B) at (1, 1.5, \topt){};
\draw[wire] (2.5,1, 0) to [out=up, in=\dl] (A) to [out=\ur, in=\dl] (B) to [out=\ur, in=down] (0.2,1,1.5*\h);
\draw[wire] (B) to [out=\ur, in=down] (0.2, 2, 1.5*\h);
\draw[wire] (A) to [out=\ur, in=down, looseness=1] (0.2,0.,1.5*\h);
\node[dot] at (A){};
\node[tmor, above left] at (A) {$f$};
\node[tmor, below right] at (B) {$f$};
\end{tz}
\picgap=\picgap
\begin{tz}[td,scale=1,xscale=1.75,yscale=1]
\begin{scope}[xyplane=0]
\draw[slice] (0,0) to [out=up, in=\dl] (1,2) to (1,3);
\draw[slice] (1,0) to [out=up, in=\dl] (1.5,1) to [out=up, in=\dr] (1,2);
\draw[slice] (2,0) to [out=up, in=\dr] (1.5,1);
\end{scope}
\begin{scope}[xyplane=\top, on layer=superfront]
\draw[slice] (0,0) to [out=up, in=\dl] (1,2) to (1,3);
\draw[slice] (1,0) to [out=up, in=\dl] (1.5,1) to [out=up, in=\dr] (1,2);
\draw[slice] (2,0) to [out=up, in=\dr] (1.5,1);
\end{scope}
\begin{scope}[xzplane=0, on layer =superfront]
\draw[slice,short] (0,0) to (0, \top);
\draw[slice,short] (1,0) to (1,\top);
\draw[slice,short] (2,0) to (2,\top);
\end{scope}
\begin{scope}[xzplane=3]
\draw[slice,short] (1,0) to (1,\top);
\end{scope}
\begin{scope}[xzplane=2]
\draw[slicem,short] (1,0) to (1,\top);
\end{scope}
\begin{scope}[xzplane=1]
\draw[slicem,short] (1.5,0) to (1.5,\top);
\end{scope}
\coordinate (A) at (2,1,0.5*\h);
\node[dot] (B) at (1, 1.5, \topt){};
\draw[wire] (2.5,1, 0) to [out=up, in=\dl] (A) to [out=\ur, in=\dl] (B) to [out=\ur, in=down] (0.2,1,1.5*\h);
\draw[wire] (B) to [out=\ur, in=down] (0.2, 2, 1.5*\h);
\draw[wire] (A) to [out=\ur, in=down, looseness=1] (0.2,0.,1.5*\h);
\node[dot] at (A){};
\node[tmor] at (2.2, 1.1, 0.5*\h){$u$};
\node[tmor] at (0.8, 1.4, 0.7*\h) {$u$};
\node[dot] (L)at (2, 1, 0.6*\h){};
\node[tmor] at (2.2, 1.1, 0.6*\h) {$p$};
\node[dot] (R)at (1, 1.5, 0.8*\h){};
\node[tmor] at (0.8, 1.4, 0.84*\h) {$p$};
\end{tz}
\picgap=\picgap
\begin{tz}[td,scale=1,xscale=1.75,yscale=1]
\begin{scope}[xyplane=0]
\draw[slice] (0,0) to [out=up, in=\dl] (1,2) to (1,3);
\draw[slice] (1,0) to [out=up, in=\dl] (1.5,1) to [out=up, in=\dr] (1,2);
\draw[slice] (2,0) to [out=up, in=\dr] (1.5,1);
\end{scope}
\begin{scope}[xyplane=\top, on layer=superfront]
\draw[slice] (0,0) to [out=up, in=\dl] (1,2) to (1,3);
\draw[slice] (1,0) to [out=up, in=\dl] (1.5,1) to [out=up, in=\dr] (1,2);
\draw[slice] (2,0) to [out=up, in=\dr] (1.5,1);
\end{scope}
\begin{scope}[xzplane=0, on layer =superfront]
\draw[slice,short] (0,0) to (0, \top);
\draw[slice,short] (1,0) to (1,\top);
\draw[slice,short] (2,0) to (2,\top);
\end{scope}
\begin{scope}[xzplane=3]
\draw[slice,short] (1,0) to (1,\top);
\end{scope}
\begin{scope}[xzplane=2]
\draw[slicem,short] (1,0) to (1,\top);
\end{scope}
\begin{scope}[xzplane=1]
\draw[slicem,short] (1.5,0) to (1.5,\top);
\end{scope}
\coordinate (A) at (2,1,0.5*\h);
\node[dot] (B) at (1, 1.5, \topt){};
\draw[wire] (2.5,1, 0) to [out=up, in=\dl] (A) to [out=\ur, in=\dl] (B) to [out=\ur, in=down] (0.2,1,1.5*\h);
\draw[wire] (B) to [out=\ur, in=down] (0.2, 2, 1.5*\h);
\draw[wire] (A) to [out=\ur, in=down, looseness=1] (0.2,0.,1.5*\h);
\node[dot] at (A){};
\node[tmor] at (2.2, 1.1, 0.5*\h){$u$};
\node[tmor] at (0.8, 1.4, 0.7*\h) {$u$};
\node[dot] (L)at (2, 1, 1.1*\h){};
\node[tmor] at (2.2, 1.1, 1.1*\h) {$p$};
\node[dot] (R)at (1, 1.5, 1.1*\h){};
\node[tmor] at (1.15, 1.5, 1.1*\h) {$p$};
\end{tz}
 \]
 In particular, the composites $p^R$ and $u^R$ of equation~\eqref{eq:pR} are depicted as follows:
  \def\top{1.5*\h}
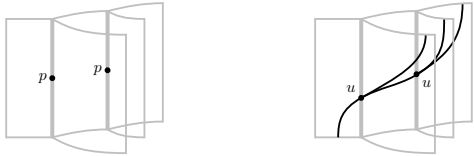
 \[\begin{tz}[td,scale=1,xscale=1.75,yscale=1,scale=0.7]
\begin{scope}[xyplane=0]
\draw[slice] (0,0) to [out=up, in=\dl] (1,2) to (1,3);
\draw[slice] (1,0) to [out=up, in=\dl] (1.5,1) to [out=up, in=\dr] (1,2);
\draw[slice] (2,0) to [out=up, in=\dr] (1.5,1);
\end{scope}
\begin{scope}[xyplane=\top, on layer=superfront]
\draw[slice] (0,0) to [out=up, in=\dl] (1,2) to (1,3);
\draw[slice] (1,0) to [out=up, in=\dl] (1.5,1) to [out=up, in=\dr] (1,2);
\draw[slice] (2,0) to [out=up, in=\dr] (1.5,1);
\end{scope}
\begin{scope}[xzplane=0, on layer =superfront]
\draw[slice,short] (0,0) to (0, \top);
\draw[slice,short] (1,0) to (1,\top);
\draw[slice,short] (2,0) to (2,\top);
\end{scope}
\begin{scope}[xzplane=3]
\draw[slice,short] (1,0) to (1,\top);
\end{scope}
\begin{scope}[xzplane=2]
\draw[slicem,short] (1,0) to (1,\top);
\end{scope}
\begin{scope}[xzplane=1]
\draw[slicem,short] (1.5,0) to (1.5,\top);
\end{scope}
\coordinate (A) at (2,1,0.5*\h);
\node[dot] (L)at (2, 1, 0.75*\h){};
\node[tmor,left] at (2, 1, 0.75*\h) {$p$};
\node[dot] (R)at (1, 1.5, 0.75*\h){};
\node[tmor,left] at (1, 1.5, 0.75*\h) {$p$};
\end{tz}
\hspace{2cm} \def\top{1.5*\h}
\begin{tz}[td,scale=1,xscale=1.75,yscale=1,scale=0.7]
\begin{scope}[xyplane=0]
\draw[slice] (0,0) to [out=up, in=\dl] (1,2) to (1,3);
\draw[slice] (1,0) to [out=up, in=\dl] (1.5,1) to [out=up, in=\dr] (1,2);
\draw[slice] (2,0) to [out=up, in=\dr] (1.5,1);
\end{scope}
\begin{scope}[xyplane=\top, on layer=superfront]
\draw[slice] (0,0) to [out=up, in=\dl] (1,2) to (1,3);
\draw[slice] (1,0) to [out=up, in=\dl] (1.5,1) to [out=up, in=\dr] (1,2);
\draw[slice] (2,0) to [out=up, in=\dr] (1.5,1);
\end{scope}
\begin{scope}[xzplane=0, on layer =superfront]
\draw[slice,short] (0,0) to (0, \top);
\draw[slice,short] (1,0) to (1,\top);
\draw[slice,short] (2,0) to (2,\top);
\end{scope}
\begin{scope}[xzplane=3]
\draw[slice,short] (1,0) to (1,\top);
\end{scope}
\begin{scope}[xzplane=2]
\draw[slicem,short] (1,0) to (1,\top);
\end{scope}
\begin{scope}[xzplane=1]
\draw[slicem,short] (1.5,0) to (1.5,\top);
\end{scope}
\coordinate (A) at (2,1,0.5*\h);
\node[dot] (B) at (1, 1.5, \topt){};
\draw[wire] (2.5,1, 0) to [out=up, in=\dl] (A) to [out=\ur, in=\dl] (B) to [out=\ur, in=down] (0.2,1,1.5*\h);
\draw[wire] (B) to [out=\ur, in=down] (0.2, 2, 1.5*\h);
\draw[wire] (A) to [out=\ur, in=down, looseness=1] (0.2,0.,1.5*\h);
\node[dot] at (A){};
\node[tmor,above left] at (A){$u$};
\node[tmor,below right] at (B) {$u$};
\end{tz}\]
Unitarity of $u^R$ and positivity of $p^R$ evidently follow from unitarity of $u$ and positivity of $p$, proving that $f^R = p^R u^R$ is the unique polar decomposition of $f^R$, which in turn allows us to conclude the proof of Theorem~\ref{thm:functorstep1} by applying Proposition~\ref{prop:trick}.

\bibliographystyle{initalpha}
\bibliography{UFC}
\end{document}

%% file: UFC-preamble.tex
\usepackage{amsmath}       
\usepackage{amsthm}        
\usepackage{amssymb}       
\usepackage{amsfonts}
\usepackage[all]{xy}
\usepackage{xspace}
\usepackage{calc}
\usepackage{pgfplots}
\usepgflibrary{fpu}
\usepackage{mathtools}
\usepackage{tabularx}
\usepackage{ifthen}		
\usepackage{graphicx}
\usepackage{docmute}
\usepackage{array}
\usepackage{subfloat} 
\usepackage{bbm}            
\usepackage{etoolbox}  
\usepackage{verbatim}
\usepackage{stmaryrd}
\usepackage{thm-restate}
\usepackage{multicol}
\usepackage{nccmath}

\usepackage[margin=1.5in]{geometry}
\DeclareRobustCommand{\SkipTocEntry}[5]{}

\usepackage{tikz}
\usetikzlibrary{matrix}
\usetikzlibrary{decorations.pathreplacing}
\usetikzlibrary{decorations.markings}
\usetikzlibrary{arrows}
\usetikzlibrary{calc}
\usetikzlibrary{shapes.misc}
\usetikzlibrary{fit}
\usepgflibrary{decorations.pathmorphing}
\usepgflibrary{shapes.geometric}
\newenvironment{tz}[1][]{%
			\begin{tikzpicture}[baseline={([yshift=-.8ex]current bounding 					box.center)},#1] %
				}{%
			\end{tikzpicture} %
			}

\makeatletter
\pgfkeys{%
  /tikz/on layer/.code={
    \pgfonlayer{#1}\begingroup
    \aftergroup\endpgfonlayer
    \aftergroup\endgroup
  },
  /tikz/node on layer/.code={
     \gdef\node@@on@layer{%
      \setbox\tikz@tempbox=\hbox\bgroup\pgfonlayer{#1}\unhbox\tikz@tempbox\endpgfonlayer\egroup}
    \aftergroup\node@on@layer
  },
  /tikz/end node on layer/.code={
    \endpgfonlayer\endgroup\endgroup
  }
}

\def\node@on@layer{\aftergroup\node@@on@layer}

\makeatletter
\def\calign@preamble{%
   &\hfil\strut@
    \setboxz@h{\@lign$\m@th\displaystyle{##}$}%
    \ifmeasuring@\savefieldlength@\fi
    \set@field
    \hfil
    \tabskip\alignsep@
}
\let\cmeasure@\measure@
\patchcmd\cmeasure@{\divide\@tempcntb\tw@}{}{}{}
\patchcmd\cmeasure@{\divide\@tempcntb\tw@}{}{}{}
\patchcmd\cmeasure@{\ifodd\maxfields@
  \global\advance\maxfields@\@ne
  \fi}{}{}{}    
  
\makeatother
\tikzset{
    master/.style={
        execute at end picture={
            \coordinate (lower right) at (current bounding box.south east);
            \coordinate (upper left) at (current bounding box.north west);
        }
    },
    slave/.style={
        execute at end picture={
            \pgfresetboundingbox
            \path (upper left) rectangle (lower right);
        }
    }
}

\tikzset{blob/.style={draw, circle, fill=white, inner sep=1pt, minimum width=15pt, font=\scriptsize, line width=0.7pt}}
\tikzset{greenregion/.style={fill=green, fill opacity=0.3, draw=none}}
\tikzset{redregion/.style={fill=red, fill opacity=0.3, draw=none}}
\tikzset{blueregion/.style={fill=blue, fill opacity=0.3, draw=none}}
\tikzset{yellowregion/.style={fill=yellow, fill opacity=0.5, draw=none}}
\tikzset{cyanregion/.style={fill=cyan, fill opacity=0.3, draw=none}}
\tikzset{orangeregion/.style={fill=orange, fill opacity=0.6, draw=none}}
\tikzset{solidgreenregion/.style={fill=green!30, fill opacity=1, draw=none}}
\tikzset{solidredregion/.style={fill=red!30, fill opacity=1, draw=none}}
\tikzset{solidblueregion/.style={fill=blue!30, fill opacity=1, draw=none}}
\tikzset{solidyellowregion/.style={fill=yellow!30, fill opacity=1, draw=none}}
\tikzset{string/.style={line width=0.7pt}}
\tikzset{zig/.style={decoration={zigzag,segment length=3, amplitude=0.5}}}
\tikzset{bnd/.style={draw,string}}   
\tikzset{projector/.style={circle, draw, font=\scriptsize, inner sep=-5pt, minimum width=0.35cm, string, fill=white}}
\tikzset{dimension/.style={font=\scriptsize, inner sep=1pt}}

\renewcommand{\to}[1][]{\ensuremath{\xrightarrow{#1}}}
\newcommand{\To}[1][]{\ensuremath{\xRightarrow{#1}}}

\allowdisplaybreaks[1]

\theoremstyle{plain} 
\newtheorem{maintheorem}{Theorem}
\newtheorem{maincor}[maintheorem]{Corollary}

\newtheorem{mainquestion}[maintheorem]{Question}
\newtheorem{theorem}{Theorem}[section]

\newtheorem{corollary}[theorem]{Corollary}          
              
\newtheorem{cor}[theorem]{Corollary}          
\newtheorem{prop}[theorem]{Proposition}

\theoremstyle{definition} 
\newtheorem{definition}[theorem]{Definition}

\theoremstyle{remark}  
\newtheorem{mainremark}[maintheorem]{Remark}
\newtheorem{remark}[theorem]{Remark}

\newtheorem*{guide*}{Guide}
\newtheorem*{outline*}{Outline}
\newtheorem*{remarkohc*}{Remark on higher categories and the cobordism hypothesis}
\newtheorem*{assumpfield*}{Assumptions on the base field}

\newtheoremstyle{special_statement} 
	{\topskip}
	{\topskip}
	{\addtolength{\leftskip}{2.5em} \itshape }
	{}
	{\bfseries}
	{:}
	{.5em}
	{}
\theoremstyle{special_statement}

\newcommand{\nid}{\noindent}

\DeclareMathOperator{\Hom}{Hom}
\DeclareMathOperator{\End}{End}

\newcommand{\id}{\mathrm{id}}

\newcommand{\Aut}{\mathrm{Aut}}

\newcommand{\Vect}{\mathrm{Vect}}

\newcommand{\tHilb}{\mathrm{2Hilb}}

\newcommand{\op}{\mathrm{op}}

\newcommand{\Hilb}{\mathrm{Hilb}}

\def\cA{\mathcal A}\def\cB{\mathcal B}\def\cC{\mathcal C}\def\cD{\mathcal D}

\def\cM{\mathcal M}\def\cN{\mathcal N}

\def\cU{\mathcal U}
\def\cZ{\mathcal Z}

\def\CC{\mathbb C}
\def\FF{\mathbb F}\def\GG{\mathbb G}

\def\PP{\mathbb P}
\def\RR{\mathbb R}

\makeatletter
\DeclareFontFamily{OMX}{MnSymbolE}{}
\DeclareSymbolFont{MnLargeSymbols}{OMX}{MnSymbolE}{m}{n}
\SetSymbolFont{MnLargeSymbols}{bold}{OMX}{MnSymbolE}{b}{n}
\DeclareFontShape{OMX}{MnSymbolE}{m}{n}{
    <-6>  MnSymbolE5
   <6-7>  MnSymbolE6
   <7-8>  MnSymbolE7
   <8-9>  MnSymbolE8
   <9-10> MnSymbolE9
  <10-12> MnSymbolE10
  <12->   MnSymbolE12
}{}
\DeclareFontShape{OMX}{MnSymbolE}{b}{n}{
    <-6>  MnSymbolE-Bold5
   <6-7>  MnSymbolE-Bold6
   <7-8>  MnSymbolE-Bold7
   <8-9>  MnSymbolE-Bold8
   <9-10> MnSymbolE-Bold9
  <10-12> MnSymbolE-Bold10
  <12->   MnSymbolE-Bold12
}{}

\let\llangle\@undefined
\let\rrangle\@undefined
\DeclareMathDelimiter{\llangle}{\mathopen}%
                     {MnLargeSymbols}{'164}{MnLargeSymbols}{'164}
\DeclareMathDelimiter{\rrangle}{\mathclose}%
                     {MnLargeSymbols}{'171}{MnLargeSymbols}{'171}
\makeatother

\usepackage[pdftex, colorlinks=true, linkcolor=black, citecolor=black, urlcolor=black, hypertexnames=false
	]{hyperref}

\newcommand{\iso}{\cong}

\newcommand{\EQ}{\underline{\mathrm{Eq}}}
\newcommand{\Eq}{\mathrm{Eq}}
\newcommand{\EQBR}{\underline{\mathrm{EqBr}}}
\newcommand{\BrPic}{\underline{\mathrm{BrPic}}}
\newcommand{\ignore}[1]{}

\usetikzlibrary{3d}
\def\dl{-135}
\def\dr{-45}
\def\ul{135}
\def\ur{45}
\def\h{3}
\tikzset{td/.style={
					y={(0.4cm, 0.6cm)},
					x={(-1cm, 0cm)},					
                    z  = {(0cm,1cm)},
                    scale = 0.5}}
\tikzset{slice/.style={draw = gray!50, line width = 0.7pt}}

\tikzset{xyplane/.style={canvas is yx plane at z=#1}}
\tikzset{xzplane/.style={canvas is yz plane at x=#1}}
\tikzset{yzplane/.style={canvas is xz plane at y=#1}}

\pgfdeclarelayer{back}
\pgfdeclarelayer{front}
\pgfdeclarelayer{superfront}
\pgfdeclarelayer{superback}
\pgfdeclarelayer{fronta}
\pgfdeclarelayer{frontb}
\pgfdeclarelayer{frontc}
\pgfdeclarelayer{backa}
\pgfsetlayers{superback,back,backa,main,fronta,frontb,frontc,front,superfront}
\tikzset{short/.style={shorten >=-0.26*\lw,shorten <=-0.25*\lw}}    
\def\lw{1.4pt}  
\tikzset{wire/.style={black, line width=0.7pt}}
\tikzset{dot/.style={circle, scale=0.15, fill=black, thick, draw}}
\tikzset{tmor/.style={scale=0.6, color=black}}   
\tikzset{obj/.style={scale=0.6, color=gray}}
\tikzset{omor/.style={scale=0.6, color=black}}   

\newcommand{\lact}{\vartriangleright}